\documentclass[12pt,a4paper]{article}

\usepackage{amssymb, amsmath, amsthm}

\usepackage[cm]{fullpage}
\usepackage[english]{babel}
\usepackage[pdftex]{graphicx}
\usepackage{dsfont}
\usepackage{booktabs}

\usepackage{hyperref}

\newtheorem{thm}{Theorem}
\newtheorem{lem}{Lemma}

\title{Error of the Galerkin scheme for a semilinear subdiffusion equation with time-dependent coefficients and nonsmooth data}
\author{\L ukasz P\l ociniczak\thanks{Faculty of Pure and Applied Mathematics, Wroc{\l}aw University of Science and Technology, Wyb. Wyspia{\'n}skiego 27, 50-370 Wroc{\l}aw, Poland}$\;^,$\footnote{Email: lukasz.plociniczak@pwr.edu.pl}}
\date{}

\begin{document}
\maketitle

\begin{abstract}
	We investigate the error of the (semidiscrete) Galerkin method applied to a semilinear subdiffusion equation in the presence of a nonsmooth initial data. The diffusion coefficient is allowed to depend on time. It is well-known that in such parabolic problems the spatial error increases when time decreases. The rate of this time-dependency is a function of the fractional parameter and the regularity of the initial condition. We use the energy method to find optimal bounds on the error under weak and natural assumptions on the diffusivity. First, we prove the result for the linear problem and then use the "frozen nonlinearity" technique coupled with various generalizations of Gr\"onwall inequality to carry the result to the semilinear case. The paper ends with numerical illustrations supporting the theoretical results. \\
	
	
\end{abstract}

\section{Introduction}
We consider the following semilinear subdiffusion equation on a convex domain $\Omega \subset \mathbb{R}^d$ ($d\geq 1$) with smooth boundary $\partial \Omega$, initial data $\varphi \in L^2(\Omega)$, and Dirichlet boundary condition
\begin{equation}
	\label{eqn:MainPDE}
	\begin{cases}
		\partial^\alpha_t u = \nabla\cdot\left(D(x,t) \nabla u\right) + f(x, t, u), & x\in\Omega, \quad t \in (0,T], \quad \alpha\in(0,1), \\
		u(x,0) = \varphi(x), & x\in\Omega,\\
		u(x,t) = 0, & x\in\partial\Omega, \quad t \in (0,T],\\
	\end{cases}
\end{equation}
where $T>0$ is the fixed time horizon. The evolution in time is governed by the Caputo fractional derivative
\begin{equation}
\label{eqn:Caputo}
	\partial^\alpha_t u(x,t) = \frac{1}{\Gamma(1-\alpha)} \int_0^t (t-s)^{-\alpha} u_s(x,s)ds, \quad \alpha \in (0,1). 
\end{equation}
This fractional derivative can also be written in terms of the fractional integral
\begin{equation}
	\label{eqn:FractionalIntegral}
	I^{\alpha}_t u(x,t) = \frac{1}{\Gamma(\alpha)} \int_0^t (t-s)^{\alpha-1} u(x,s) ds, \quad I_t := I^1_t, \quad \alpha >0,
\end{equation}
as
\begin{equation}
	\partial^\alpha_t u(x,t) := I^{1-\alpha}_t u_t(x,t).
\end{equation}
We assume that the diffusivity $D$ and the source $f$ are smooth functions satisfying
\begin{equation}
\label{eqn:Assumptions}
\begin{split}
	0< D_- \leq D(x,t) \leq D_+, \quad &|D_t(x,t)| \leq \kappa(t) \in L^1((0,T)), \quad D_x(\cdot,t) \in L^2(\Omega), \\
	|f(x,t,u)|\leq F, \quad &|f(x,t,u)-f(x,t,v)| \leq L |u-v|,
\end{split}
\end{equation}
for some positive constants $D_\pm$, $F$, and $L$. We thus require only boundedness and regular behaviour of the nonlinearities. In particular we assume that the time derivative of the diffusivity is only $t-$integrable and $x$-bounded. This is a natural and much weaker assumption than imposed in previous studies \cite{Mus18, Jin19a}, where boundedness on first and second partial derivatives were required. 

There is a strong motivation to consider analytics and numerics for problems of type (\ref{eqn:MainPDE}). Equations of this form, that is where the dependence on time exhibits some memory effects, arise in modelling of many interesting phenomena. For example, a subdiffusive behaviour appears when a randomly walking particle slows down in the sense that its mean-square displacement is sublinear \cite{Met00,Kla08}. A real-world example of this behaviour is water imbibition in certain porous media where complex matrix structure or chemical reactions happening there may slow down the evolution \cite{Plo15,Plo14,Plo19a,El04,El20}. Some other places where the subdiffusion emerges is for example: turbulent flow \cite{Hen02}, material science \cite{Mul96}, viscoelasticity \cite{Amb96}, cell biochemistry \cite{Sun17}, biophysics \cite{Kou08}, and plasma physics \cite{Del05}. An interesting collection of various applications of fractional calculus (not only subdiffusion) can be found in \cite{sun2018new}. 

To see the importance of weak regularity assumptions on $D$ we have to discuss the smoothness of the solution to (\ref{eqn:MainPDE}). Usually it is nonsmooth in time near $t= 0$ which is a very well-known fact for the linear diffusion with or without time-dependent coefficients (see for ex. \cite{Sak11} for the seminal result and \cite{Jin19} for a review). For example, the constant coefficient homogeneous equation has a solution expressible in terms of the Mittag-Leffler function that close to $t=0$ behaves as
\begin{equation}
	u(x,t) \sim \varphi(x) + a(x) t^\alpha, \quad t\rightarrow 0^+,
\end{equation}   
for some function $a$. Therefore, the solution has a weak singularity at the temporal origin. More precisely, for the linear homogeneous equation with constant coefficients we have \cite{Sak11,McL10}
\begin{equation}
	\|\partial^{(m)}_t u\|_q \leq C t^{-\frac{\alpha(q-p)}{2}-m} \|\varphi\|_p, \quad 0\leq p- q \leq 2.
\end{equation}
This is in strong contrast with the classical parabolic equation for which we can take $p\geq q$ (see \cite{Tho07}). Although the constant- or only space-dependent coefficient case is very well studied, the time-dependent, let alone semilinear, has started to be investigated only recently. For example, the above regularity estimate has been proved in \cite{Jin19} by perturbation techniques for the special case $q=2$, $m=1$, $f=f(x,t)$, and diffusivity $D(x,t)$ with bounded first and mixed derivatives. Some further results concerning existence, uniqueness, and regularity of solutions to time-dependent and quasilinear problems with essentially minimal assumptions can be found in \cite{Ver15,Zac12, Kub18}. The nonhomogeneous situation is similar: initial singularity can be avoided only if we assume certain regularity of $f(\cdot, 0, \cdot)$ and its time derivatives \cite{Jin19}. Therefore, following \cite{Mus18} it is well-motivated to assume that the solution of (\ref{eqn:MainPDE}) satisfies 
\begin{equation}
	\label{eqn:Regularity}
	\|u(t)\|_q + t \|u'(t)\|_q \leq C t^{-\frac{\alpha(q-p)}{2}} \|\varphi\|_p, \quad 0\leq p- q \leq 2.
\end{equation}
Since the linear case satisfies the above estimate it is sensible to assume that the solution of our equation has at least the same regularity. We will use that in formulating the Galerkin method. It is now possible to understand our claim that the assumption on $t-$integrability of $D_t$ is natural. Suppose that we would like to numerically solve the \emph{quasilinear} equation with $D=D(x,t,u)$ (to be considered in future work). Then, if the exact solution is $u=u(x,t)$ we can form a new auxiliary diffusivity $\widetilde{D}(x,t) := D(x,t,u(x,t))$. By the chain rule we have $\widetilde{D}_t = D_t + D_u u_t$ and even if $D$ were very smooth, the diffusivity $\widetilde{D}$ would not be such - thanks to (\ref{eqn:Regularity}). In particular, we cannot expect for it to be bounded and we are not permitted to use previous results from the literature. On the other hand, the function $\widetilde{D}_t$ is $t$-integrable and our assumptions (\ref{eqn:Assumptions}) are satisfied. 

As until the time of writing this paper, the majority of results concerning subdiffusion problems investigated mostly either constant or space-dependent diffusivity in both analytical and numerical settings. There is a vast literature on these topics and we mention only several of them. In \cite{Sak11} a complete characterisation of solutions to the constant diffusivity case with a possibly nonsmooth initial condition was given. The numerical treatment and Galerkin method error estimates was carried over for example in \cite{Jin13}. The case of only space-dependent diffusivity has been studied for example in \cite{McL10} from analytical and in \cite{eidelman2004cauchy} from numerical points of view. Many fully discrete schemes were proposed to solve the subdiffusion equation with \emph{time-independent} diffusivity. The Caputo derivative can be discretized in different ways, from which probably the convolution quadrature \cite{Lub04} and the L1 method \cite{Old74} are the most popular. Some recent developments in these fully discrete schemes were given in \cite{Jin16,Cue06} (convolution quadrature) and \cite{jin2016analysis, Kop19, stynes2017error, Lia18} (L1 method). Some interesting results concerning the L1 scheme for the semilinear subdiffusion equation has been published in \cite{al2019numerical} where the nonsmoothness of the initial data has been taken into account. We would also like to mention our previous integral equation approach to devise numerical methods for (\ref{eqn:MainPDE}) in the case of degenerate porous medium equation, that is when $D(x,t,u) \approx u^{m}$ with $m>1$ (see \cite{Plo19,Plo17,okrasinska2022second}). The reader is also invited to consult \cite{Jin19} (and references therein) for a concise survey of numerical methods in the presence of nonsmooth data. 

The situation is different when we allow the diffusion coefficient to depend on time. The frequently used tools such as separation of variables or Laplace transform cannot be used in that case to analyse solutions and their regularity. Instead, energy methods provide a very versatile approach \cite{Ver15,Zac12}. Some approaches also use perturbation techniques \cite{kim2017lq}. The Galerkin method for this time-dependent case was analysed in few papers. The perturbation technique was applied in \cite{Jin19a,jin2020subdiffusion} along with extensive numerical simulations using convolution quadrature. On the other hand, the energy method was applied in \cite{Mus18} to obtain optimal error estimates. In all of these approaches authors assumed boundedness of the first and mixed derivatives of the diffusion coefficient. To relax these assumption is the main motivation of this paper. Moreover, we are also able to extend the optimal error bounds into the semilinear case. 

The paper is structured as follows. The next section contains our main results. It starts with necessary preliminaries where we cite all the literature results that are needed in the sequel. Then, we move to the linear case and prove the nonsmooth optimal error estimate. The section ends with an extension to the semilinear equation with a use of the frozen nonlinearity technique. In Section 3 we collect some numerical results that support our theoretical considerations. 

\section{Error for the Galerkin scheme}

\subsection{Preliminaries}
In this section we collect several known results that will be needed in below considerations. 

By $L^2(\Omega)$ and $H^s(\Omega)$ we denote Hilbert and Sobolev spaces with the usual norms $\|\cdot\|_s$. We also write $\|\cdot\| := \|\cdot\|_0$. The subspace of $H^1$ of functions with vanishing trace is denoted by $H_0^1(\Omega)$ with the norm inherited due to the Poincar\'e-Friedrichs inequality. In regularity theory some additional function spaces are much useful. Let $\{\lambda_n\}_{n=1}^\infty$ denote the eigenvalues of the negative Laplacian with vanishing Dirichlet boundary conditions on $\partial\Omega$. The corresponding eigenfunctions normalized with respect to the $L^2(\Omega)$ norm are denoted by $\{\phi_n\}_{n=1}^\infty$. Then, for any $s\geq 0$ the space $\dot{H}^s (\Omega) \subseteq L^2(\Omega)$ contains precisely these functions $v$ that satisfy $\|v\|_s < \infty$, where \cite{Tho07}
\begin{equation}
\label{eqn:HDotNorm}
	\|v\|_s^2 := \|(-\Delta)^\frac{s}{2} v\| = \sum_{n=0}^\infty \lambda_n^s (v,\phi_n)^2.
\end{equation}
Notice that for $0\leq s < 1/2$ we have $\dot{H}^s(\Omega) = H^s(\Omega)$ while for $1/2 < s \leq 2$ it holds that $\dot{H}^s(\Omega) = \{v\in H^s: \, v = 0 \text{ on } \partial\Omega\}$. In particular, $\dot{H}^1(\Omega) = H_0^1 (\Omega)$ and $\dot{H}^2(\Omega) = H_0^1(\Omega) \cap H^2(\Omega)$.

Further, we recall some results concerning fractional derivatives. Th first one is a nonlocal version of the derivative of a product rule. In general it yields an infinite sum but for the polynomial case it simplifies considerably \cite{Li19}. For any sufficiently smooth $y=y(t)$ we have
\begin{equation}
	\label{eqn:Product}
	\partial^\alpha_t (t y(t)) = t \partial^\alpha_t y(t) + \alpha I^{1-\alpha}_t y(t) + \frac{t^{1-\alpha}}{\Gamma(1-\alpha)} y(0).
\end{equation}
The next elementary identity concerns the integral of a Caputo derivative. It easily follows from the definition (\ref{eqn:Caputo}) that
\begin{equation}
	\label{eqn:IntegralDerivative}
	I_t \partial^\alpha_t y(t) = I^{1-\alpha}_t y(t) - \frac{t^{1-\alpha}}{\Gamma(2-\alpha)} y(0), \quad I^\alpha_t \partial^\alpha_t y(t) = y(t) - y(0).
\end{equation}
Additionally, the fractional integral enjoys some very useful positivity and continuity properties that are inherited from the classical integral. The ones that will particularly be useful are (proofs can be found in \cite{mustapha2014well}, Lemma 3.1)
\begin{equation}
\label{eqn:FracIntContPositiv}
\begin{split}
	\left| \int_0^t (I^{1-\alpha} u(s), v(s)) ds \right| \leq \epsilon \int_0^t (I^{1-\alpha} u(s), u(s)) ds + \frac{1}{4\epsilon (1-\alpha)} \int_0^t (I^{1-\alpha} v(s), v(s)) ds, \\
	\int_0^t (I^{1-\alpha} u(s), u(s)) ds \geq 0, \quad 0 < \alpha < 1,
\end{split}
\end{equation}
where $u=u(x,t)$ and $v=v(x,t)$ are $t$-piecewise continuous. We also have a connection between a derivative and the fractional integral (see \cite{le2016numerical}, Lemma 2.1)
\begin{equation}
\label{eqn:FracIntDer}
	\int_0^t (I^{1-\alpha} u_t(s), u_t(s)) ds \geq C t^{-\alpha} \|u(s)\|, \quad 0 < \alpha < 1,
\end{equation} 
provided that $u(x,0) = 0$ and $u_t$ is piecewise continuous. Note that in general, above inequalities are not $\alpha$-robust, that is the various constants blow up as we approach the classical cases $\alpha\rightarrow 0^+$ and $\alpha\rightarrow 1^-$. However, the proofs for these are straightforward. Lastly, we will need two lemmas concerning inequalities with fractional derivatives. The first one gives a relationship between the inner product and the norm involving a Caputo derivative. It is a nontrivial generalization of the classical result. The following result has been recently proved in \cite{Ver15,Kop22} but for completeness we reprint here its short proof. 
\begin{lem}\label{lem:DerivativeScalar}[\cite{Kop22}, Lemma 2.8; \cite{Ver15}]
	Assume that $v \in L^\infty ((0,T); L^2(\Omega)) \cap W^{1,\infty} ((\epsilon, T); L^2(\Omega))$ for any $0<\epsilon< T$. If $v(0) = 0$ then 
	\begin{equation}
		(\partial^\alpha_t v, v) \geq \|v\| (\partial^\alpha_t \|v \|).
	\end{equation}
\end{lem}
\begin{proof}
	The idea behind the proof is to regularize the integrand in (\ref{eqn:Caputo}) in order to integrate by parts. Since $(v(\cdot, s))_s = (v(\cdot, s) - v(\cdot, t))_s$ we have
	\begin{equation}
		\begin{split}
			\partial^\alpha_t v(x,t) 
			&= \frac{1}{\Gamma(1-\alpha)} \int_0^t (t-s)^{-\alpha} (v(x,s)- v(x,t))_s ds \\
			&= \frac{1}{\Gamma(1-\alpha)}\left(t^{-\alpha} v(x,t) - \alpha \int_0^t (t-s)^{-\alpha-1} (v(x,s)- v(x,t)) ds\right).
		\end{split}
	\end{equation}
	where in calculating the first term we used the initial condition for $v$ and the fact that its derivative is bounded. Also, due to this reason the integral is well-defined. Now, if we take the inner product of the above with $v(t)$ we obtain
	\begin{equation}
		\begin{split}
			(\partial^\alpha v, v) 
			&= \frac{1}{\Gamma(1-\alpha)}\left(t^{-\alpha} \|v(t)\|^2 - \alpha \int_0^t (t-s)^{-\alpha-1}((v(s), v(t))- \|v(t)\|^2) ds\right) \\
			&= \frac{\|v(t)\|}{\Gamma(1-\alpha)}\left(t^{-\alpha} \|v(t)\| - \alpha \int_0^t (t-s)^{-\alpha-1}\left(\frac{(v(s), v(t))}{\|v(t)\|}- \|v(t)\|\right) ds\right).
		\end{split}
	\end{equation}
	Hence, by the Cauchy-Schwarz inequality we arrive at
	\begin{equation}
		(\partial^\alpha v, v) \geq \frac{\|v(t)\|}{\Gamma(1-\alpha)}\left(t^{-\alpha} \|v(t)\| - \alpha \int_0^t (t-s)^{-\alpha-1}\left(\|v(s)\|- \|v(t)\|\right) ds\right) = \|v(t)\| (\partial^\alpha_t \|v(t)\|),
	\end{equation}
	which follows by integrating back by parts. 
\end{proof}
The next lemmas are generalizations of the Gr\"onwall's inequality: one to fractional integrals in which we allow for certain singularity in the free term, and the other to integrable kernel in the classical case.
\begin{lem}[Fractional Gr\"onwall inequality, \cite{Hen06,Web19}]
	\label{lem:Gronwall}
	Let $y(t) \geq 0$ be continuous for $0<t\leq T$. Assume that
	\begin{equation}
		y(t) \leq A t^{-\beta} + B I^\alpha y(t), \quad \alpha\in (0,1),
	\end{equation}
	with $A\geq 0$, $B>0$, and $\beta \in [0,1)$. Then, there exist a constant $C=C(\alpha,\beta,B,T)>0$ such that
	\begin{equation}
		y(t) \leq C A t^{-\beta}. 
	\end{equation}
\end{lem}
Note that in the below lemma we do not assume anything about the sign of $y$ and $F$. 
\begin{lem}[Generalized classical Gr\"onwall inequality, \cite{corduneanu2008principles}]
	\label{lem:GronwallGeneralized}
	Suppose $\kappa$ is non-negative and integrable. Then, for any real-valued functions $y$ and $F$ the inequality 
	\begin{equation}
		y(t) \leq F(t) + \int_0^t \kappa(s) y(s) ds,
	\end{equation}
	implies
	\begin{equation}
		y(t) \leq F(t) + \int_0^t \kappa(s) \exp\left(\int_s^t \kappa(z) dz\right) F(s) ds.
	\end{equation}
\end{lem}

We can now proceed to the numerical method. We set up the finite element Galerkin scheme by introducing a family $\mathcal{T}_h$ of shape-regular quasi-uniform triangulations of $\overline{\Omega}$ with the maximal diameter $h = \max_{K\in\mathcal{T}_h} \text{diam}K$. Moreover, let $V_h \subset H_0^1(\Omega)$ be the standard continuous piecewise linear function space over $\mathcal{T}_h$ that vanish on the boundary $\partial\Omega$, that is
\begin{equation}
	V_h := \left\{\chi_h \in C^0(\overline{\Omega}): \; \chi_h|_K \text{ is linear for all } K\in\mathcal{T}_h \text{ and } \chi_h|_{\partial\Omega} = 0 \right\}.
\end{equation}

The weak form of the main parabolic problem is given in the standard way by multiplication by a test function and integration by parts
\begin{equation}
\label{eqn:MainPDEWeak}
	(\partial^\alpha_t u, \chi) + a(D(t); u, \chi) = (f(t, u), \chi), \quad u(0) = \varphi, \quad \chi \in H^1_0(\Omega),
\end{equation}
where we have suppressed writing the $x$-variable in inner products. The bilinear form $a$ is given by
\begin{equation}
	a(w; u, v) = \int_\Omega w(x) \nabla u \cdot \nabla v \, dx.
\end{equation}
The Galerkin approximation $u_h$ to the solution of (\ref{eqn:MainPDEWeak}) is weighted over the finite element space $V_h$
\begin{equation}
\label{eqn:MainPDEGalerkin}
	(\partial^\alpha_t u_h, \chi) + a(D(t); u_h, \chi) = (f(t, u_h), \chi), \quad u_h(0) = P_h \varphi, \quad \chi \in V_h,
\end{equation}
where we use the orthogonal projection operator $P_h$ onto the space $V_h$ defined on $L^2(\Omega)$ as
\begin{equation}
\label{eqn:Projection}
	(v - P_h v, \chi) = 0, \quad \chi\in V_h.
\end{equation}
We have the standard error estimates for sufficiently regular functions \cite{Tho07}
\begin{equation}
\label{eqn:ProjectionError}
	\|v - P_h v\| \leq C h^r \|v\|_r, \quad v \in \dot{H}^r, \quad r=1,2.
\end{equation}
In error analysis the following Ritz projection $R_h$ of $H_0^1$ onto $V_h$ is very useful
\begin{equation}
\label{eqn:Ritz}
	a(D(t, u); v - R_h v, \chi) = 0, \quad \chi \in V_h,
\end{equation} 
where $u$ is a solution of (\ref{eqn:MainPDEWeak}). Standard estimates on the projection error are similar to ones concerning $P_h$
\begin{equation}
\label{eqn:RitzError}
	\| v - R_h v \| \leq C  h^r \|v\|_r, \quad v \in \dot{H}^r, \quad r=1,2.
\end{equation}

\subsection{Linear case}
Now, we are in position to consider the time-dependent coefficient linear case. That is, we assume that the source in (\ref{eqn:MainPDEWeak}) is given by
\begin{equation}
\label{eqn:LinearDiffusivity}
	f = g(x,t),
\end{equation}
and thus we are looking for $v$ satisfying
\begin{equation}
\label{eqn:LinearPDEWeak}
	(\partial^\alpha_t v, \chi) + a(D(t); v, \chi) = (g(t), \chi), \quad v(0) = \varphi, \quad \chi \in H^1_0(\Omega), \quad 
\end{equation}
satisfying a linear equation. Adequately, the Galerkin approximation is
\begin{equation}
\label{eqn:LinearGalerkin}
	(\partial^\alpha_t v_h, \chi) + a(D(t); v_h, \chi) = (g(t), \chi), \quad v_h(0) = P_h \varphi, \quad \chi \in V_h.
\end{equation}
The definition of the Ritz projection (\ref{eqn:Ritz}) changes according to our particular choice of $D$ and we will retain using the same letter for denoting it. A similar remark concerns the assumptions (\ref{eqn:Assumptions}). 

We can now prove the result concerning the error of the above scheme.
\begin{thm}\label{thm:Linear}
Let $v$ be the solution of (\ref{eqn:LinearPDEWeak}) with regularity (\ref{eqn:Regularity}) and $v_h$ its Galerkin approximation calculated from (\ref{eqn:LinearGalerkin}). Then, assuming (\ref{eqn:Assumptions}) with (\ref{eqn:LinearDiffusivity}) we have
\begin{equation}
\label{eqn:NonsmoothError}
	\|v-v_h\| + h \|\nabla(v-v_h)\| \leq C h^2 t^{-\frac{\alpha(2-p)}{2}} \|\varphi\|_p, \quad 0\leq p\leq 2, \quad 0 < \alpha <1,
\end{equation}
where the constant $C$ depends on $u$ and $\alpha$. 
\end{thm}
\begin{proof}
By $e = v - v_h$ we denote the error of the approximation. For any $\chi\in V_h$ due to weak formulations (\ref{eqn:LinearPDEWeak}) and (\ref{eqn:LinearGalerkin}) we have the error equation
\begin{equation}
\begin{split}
	(\partial^\alpha_t e, \chi) + a(D(t); e, \chi) 
	&= (\partial^\alpha_t v, \chi) + a(D(t); v, \chi) - \left( (\partial^\alpha_t v_h, \chi) +  a(D(t); v_t, \chi) \right)\\
	&= (g(t), \chi) - (g(t), \chi) = 0.
\end{split}
\end{equation}
Now, we can express the error in the standard way with the use of Ritz projection (\ref{eqn:Ritz}), that is 
\begin{equation}
	e  = v - R_h v + R_h v - v_h = \rho + \theta,
\end{equation}
and thus we have to find estimates on $\rho$ and $\theta$. The former is easier since due to the Ritz projection error estimate (\ref{eqn:RitzError}) along with regularity assumption (\ref{eqn:Regularity}) we have
\begin{equation}
\label{eqn:rhoError}
	\begin{split}
		\|\rho\| + h \|\rho\|_1 &\leq C h^2 \|v\|_2 \leq t^{-\frac{\alpha(2-p)}{2}} \|\varphi\|_p, \quad 0\leq p\leq 2, \\
		\|\rho_t\| &\leq C h^2 \|v\|_2 \leq t^{-\frac{\alpha(2-p)}{2}-1} \|\varphi\|_p, \quad 0\leq p\leq 2,
	\end{split}
\end{equation}
which are the sought estimates on $\rho$. 

We now focus on bounding $\theta$. First, we find the appropriate differential equation for this quantity and then use to it estimate $\theta$ by certain derivatives and integrals of $\rho$. To this end, for any $\chi\in V_h$ we write
\begin{equation}
\begin{split}
	(\partial^\alpha_t \theta, \chi) + a(D(t); \theta, \chi) 
	&= (\partial^\alpha_t R_h v, \chi) + a(D(t); R_h v, \chi) - (\partial^\alpha_t v_h, \chi) - a(D(t); v_h, \chi) \\
	&= (\partial^\alpha_t R_h v, \chi) + a(D(t); v, \chi) - (g(t), \chi),
\end{split}
\end{equation}
where in the second equality we have used the definition of Ritz projection (\ref{eqn:Ritz}) and the equation for $v_h$, that is (\ref{eqn:LinearGalerkin}). Now, by the exact equation (\ref{eqn:LinearPDEWeak}) we arrive at
\begin{equation}
\label{eqn:ThetaEq}
	(\partial^\alpha_t \theta, \chi) + a(D(t); \theta, \chi) = (\partial^\alpha_t R_h v - \partial^\alpha_t v, \chi) + g(t, \chi) - (g(t), \chi) = - (\partial^\alpha_t \rho, \chi). 
\end{equation}
The strategy is to appropriately integrate the above equation to obtain an estimate on $\theta$. The key point, however, is that we cannot allow the initial condition $\theta(0) = R_h \varphi - P_h \varphi$ to appear in it. Since then, by (\ref{eqn:ProjectionError}) the bound would require smoothness of $\varphi$. In order to overcome this difficulty we multiply the above equation by $t$ and use the fractional derivative of a product rule (\ref{eqn:Product})
\begin{equation}
	(\partial^\alpha_t (t\theta), \chi) + a(D(t); t \theta, \chi) = - (\partial^\alpha_t (t\rho), \chi) + \alpha (I^{1-\alpha}_t \rho + I^{1-\alpha}_t \theta, \chi) - \frac{t^{1-\alpha}}{\Gamma(1-\alpha)} (\theta(0) + \rho(0), \chi),
\end{equation}
for $\chi \in V_h$. The term corresponding to the initial condition vanishes since by the definition of the orthogonal projection (\ref{eqn:Projection}) we have 
\begin{equation}
	(\theta(0) + \rho(0), \chi) = (v(0) - v_h(0), \chi) = (\varphi - P_h \varphi, \chi) = 0, \quad \chi\in V_h.  
\end{equation}
Therefore, our error equation is
\begin{equation}
	(\partial^\alpha_t (t\theta), \chi) + a(D(t); t \theta, \chi) = - (\partial^\alpha_t (t\rho), \chi) + \alpha (I^{1-\alpha}_t \rho , \chi) + \alpha(I^{1-\alpha}_t \theta, \chi), \quad \chi \in V_h,
\end{equation}
which after taking $\chi = (t\theta)_t \in V_h$ and noticing that $\partial^\alpha_t = I^{1-\alpha}_t \partial^1_t$ becomes
\begin{equation}
	(I^{1-\alpha}_t (t\theta)_t,  (t\theta)_t) + (D(t)t \nabla\theta, (t\nabla\theta)_t) = - (I^{1-\alpha}_t (t\rho)_t, (t\theta)_t) + \alpha (I^{1-\alpha}_t \rho , (t\theta)_t) + \alpha(I^{1-\alpha}_t \theta, (t\theta)_t)
\end{equation}
Now, we can integrate the above and use the continuity property (\ref{eqn:FracIntContPositiv}) for an appropriate choice of $\epsilon$ to kick back the terms with $(t\theta)_t$
\begin{equation}
\label{eqn:tThetaIntEq}
\begin{split}
	\int_0^t (I^{1-\alpha}_s (s\theta)_s,  (s\theta)_s) ds &+ \int_0^t (D(s)s \nabla\theta, (s\nabla \theta)_s) \leq \frac{1}{2}\int_0^t (I^{1-\alpha}_s (s\theta)_s,  (s\theta)_s) ds \\
	&+ C\left(\int_0^t |(I^{1-\alpha}_s (s\rho)_s, (s\rho)_s)|ds + \alpha \int_0^t |(I^{1-\alpha}_s \rho , \rho)| ds + \alpha\int_0^t (I^{1-\alpha}_s \theta, \theta) ds\right).
\end{split}
\end{equation}
We would like to obtain an estimate on the $\theta$-term on the right-hand side. To this end, we go back to (\ref{eqn:ThetaEq}) and integrate to arrive at
\begin{equation}
	(I^{1-\alpha}_t \theta, \chi) + \left(\int_0^t D(s) \nabla \theta(s) ds, \chi \right) = -(I^{1-\alpha}_t \rho, \theta) + \frac{t^{1-\alpha}}{\Gamma(2-\alpha)} (\theta(0)+\rho(0), \chi).
\end{equation}
Once again, the projection of the initial condition $\theta(0)+\rho(0) = \varphi - P_h\varphi$ vanishes due to orthogonality. We can thus put $\chi = \theta$ and integrate the second time
\begin{equation}
	\int_0^t (I^{1-\alpha}_s \theta, \theta) ds + \int_0^t \left(\int_0^s D(z) \nabla \theta(z) dz, \nabla\theta(s) \right)ds \leq \frac{1}{2}\int_0^t (I^{1-\alpha}_s \theta, \theta) ds + C\int_0^t |(I^{1-\alpha}_s \rho, \rho)| ds,
\end{equation}
where we have used the continuity property of $I^{1-\alpha}$ as in (\ref{eqn:FracIntContPositiv}) to extract the $\theta$-term. Now, the key-point is to transform the double integral to a form that permits us to integrate by parts
\begin{equation}
\begin{split}
	\int_0^t &\left(\int_0^s D(z) \nabla \theta(z) dz, \nabla\theta(s) \right)ds 
	= \int_0^t \frac{1}{2} \frac{1}{D(s)} \frac{d}{ds} \left\|\int_0^s D(z) \nabla \theta(z) dz\right\|^2 ds \\
	&= \frac{1}{2} \frac{1}{D(t)} \left\|\int_0^t D(s) \nabla \theta(s) ds\right\|^2 - \frac{1}{2}\int_0^t \left(\frac{1}{D(s)}\right)_s \left\|\int_0^s D(z) \nabla \theta(z) dz\right\|^2 ds,
\end{split}
\end{equation}
whence
\begin{equation}
	\int_0^t (I^{1-\alpha}_s \theta, \theta) ds +  \frac{1}{D_-} \left\|\int_0^t D(s) \nabla \theta(s) ds\right\|^2 \leq C \int_0^t |(I^{1-\alpha}_s \rho, \rho)| ds + \int_0^t \left|\left(\frac{1}{D(s)}\right)_s\right| \left\|\int_0^s D(z) \nabla \theta(z) dz\right\|^2 ds.
\end{equation}
Observe that due to integrability of $D_t$ the last integrand makes sense since $(D^{-1})_t = -D_t/D^2$ and $D$ is bounded. The application of the generalized Gr\"onwall inequality (Lemma \ref{lem:GronwallGeneralized}) for the gradient term helps us to obtain
\begin{equation}
\begin{split}
	\int_0^t (I^{1-\alpha}_s \theta, \theta) ds &+  \frac{1}{D_-} \left\|\int_0^t D(s) \nabla \theta(s) ds\right\|^2 \leq C R_1(t)^2 \\
	&+ D_-\int_0^t \left|\left(\frac{1}{D(s)}\right)_s\right| \exp\left(\int_s^t \left|\left(\frac{1}{D(z)}\right)_z\right| dz\right) R_1(s)^2 ds \\
	&- D_-\int_0^t \left|\left(\frac{1}{D(s)}\right)_s\right| \exp\left(\int_s^t \left|\left(\frac{1}{D(z)}\right)_z\right| dz\right) \left(\int_0^s (I^{1-\alpha}_z \theta, \theta) dz\right) ds
\end{split}
\end{equation}
where we have put the $\rho$-term into the definition of $R_1$. Since the integral with $I^{1-\alpha}$ is positive thanks to (\ref{eqn:FracIntContPositiv}) we can discard the last term on the right-hand side. The same can be done to the gradient term. Notice that $R_1(s)$ is increasing and hence,
\begin{equation}
\label{eqn:IBound}
	\begin{split}
		\int_0^t (I^{1-\alpha}_s \theta, \theta) ds 
		&\leq C \left( R_1(t)^2 + \int_0^t \left|\left(\frac{1}{D(s)}\right)_s\right| \exp\left(\int_s^t \left|\left(\frac{1}{D(z)}\right)_z\right| dz\right) R_1(s)^2 ds \right) \\
		&\leq C \left(R_1(t)^2 - R_1(t)^2 \int_0^t \frac{d}{ds} \left( \exp\left(\int_s^t \left|\left(\frac{1}{D(z)}\right)_z\right| dz \right) \right) ds \right) \\
		&= C R_1(t)^2 \exp\left(\int_0^t \left|\left(\frac{1}{D(z)}\right)_z\right| dz\right) \leq C R_1(t)^2,
	\end{split}
\end{equation}
where the last inequality follows from integrability of $D_t$.

We can now go back to (\ref{eqn:tThetaIntEq}) in which we use (\ref{eqn:IBound}) and gather all $\rho$-terms in a new (increasing) function $R_2(t)^2$ to arrive at
\begin{equation}
	C t^{-\alpha} \|t \theta\|^2 +\frac{1}{2} \int_0^t (D(s)s \nabla\theta, (s\nabla \theta)_s) \leq C R_2(t)^2,
\end{equation}
where we have use the derivative property (\ref{eqn:FracIntDer}). We would also like to simplify the gradient term. To this end, notice that $s \nabla\theta (s\nabla \theta)_s = ((s \nabla\theta)^2)_s/2$ and integrate by parts
\begin{equation}
	Ct^{-\alpha} \|t \theta\|^2  + D_-\|t \nabla\theta \|^2 \leq C R_2(s)^2 + \int_0^t \kappa(s) \|s \nabla\theta\|^2 ds,
\end{equation}
where we have used assumptions on the diffusivity (\ref{eqn:Assumptions}). Finally, application of the generalized Gr\"onwall inequality settles us with
\begin{equation}
\begin{split}
	Ct^{-\alpha} \|t \theta\|^2  + D_-\|t \nabla\theta \|^2 &\leq C \left(R_2(s)^2 + \int_0^t \kappa(s) \exp\left(\int_s^t \kappa(z) dz\right) R(s)^2 ds \right. \\
	&\left.- \int_0^t \kappa(s) \exp\left(\int_s^t \kappa(z) dz\right) s^{-\alpha} \|s \theta\|^2 ds \right).
\end{split}
\end{equation}
Discarding the $\theta$ term on the right-hand side and using the fact that $R_2$ is increasing simplifies the above into
\begin{equation}
	Ct^{-\alpha} \|t \theta\|^2  + D_-\|t \nabla\theta \|^2 \leq C R_2(t)^2. 
\end{equation}
What remains is to estimate $R_2(t)^2$ using (\ref{eqn:rhoError}). It contains two terms which are
\begin{equation}
\begin{split}
	\int_0^t |(I^{1-\alpha}_s \rho, \rho)| ds &\leq \int_0^t \|I^{1-\alpha}_s \rho\| \| \rho\| ds \\
	&\leq C h^4 \|\varphi\|^2_p \int_0^t s^{1-\alpha-\frac{\alpha(2-p)}{2}}s^{-\frac{\alpha(2-p)}{2}}ds = C h^4 t^{2-\alpha-\alpha(2-p)} \|\varphi\|^2_p ,
\end{split}
\end{equation}
and
\begin{equation}
	\int_0^t |(I^{1-\alpha}_s (s \rho)_s, (s \rho)_s)| ds \leq \int_0^t  \|I^{1-\alpha}_s (\rho + s \rho_s)\| \|\rho + s \rho_s\| ds \leq C h^4 t^{2-\alpha-\alpha(2-p)} \|\varphi\|^2_p.
\end{equation}
And therefore
\begin{equation}
	R_s(t)^2 \leq C  h^4 t^{2-\alpha-\alpha(2-p)} \|\varphi\|^2_p,
\end{equation}
which brings us to
\begin{equation}
	t^{-\alpha} \|t \theta\|^2 \leq C  h^4 t^{2-\alpha-\alpha(2-p)} \|\varphi\|^2_p,
\end{equation}
which after simplification by powers of $t$ proves our main claim. The estimate on the gradient follows from the inverse inequality (since our mesh is quasi-uniform)
\begin{equation}
	\|\nabla \theta\| \leq C h^{-1} \|\theta\|.
\end{equation}
This concludes the proof. 
\end{proof}
The above proof has several important points. The first one is a careful estimate of $\|t\theta\|$ rather than $\theta$ in order to avoid the emergence of the initial condition. Since then, the error estimates of the projections would require its smoothness. Wherever the initial condition arises, it appears inside a inner product with a function from $V_h$. The second key point is to use the definition of the orthogonal projection to conclude this product vanishes. Lastly, a careful integration by parts and use of the generalized Gr\"onwall inequality allows for extracting the gradient term without the need of bounding $D_t$. 

\subsection{Semilinear case}
We can now return to the semilinear case (\ref{eqn:MainPDEWeak}). The main idea in obtaining the error bounds is to use an intermediate linear finite dimensional problem that utilizes the frozen nonlinearity technique.
\begin{thm}\label{thm:semilinear}
Let $u$ be the solution of (\ref{eqn:MainPDEWeak}) satisfying (\ref{eqn:Regularity}) while $u_h$ its Galerkin approximation (\ref{eqn:MainPDEGalerkin}). Assuming (\ref{eqn:Assumptions}) we have
\begin{equation}
\label{eqn:NonsmoothErrorQuasilinear}
	\|u-u_h\| + h \|\nabla(u-u_h)\| \leq C h^2 t^{-\frac{\alpha(2-p)}{2}} \|\varphi\|_p, \quad 0 \leq p \leq 2, \quad 0<\alpha<1,
\end{equation}
where the constant $C$ depends on $u$ and $\alpha$.
\end{thm}
\begin{proof}
Let $v$ be the solution of the linear equation (\ref{eqn:LinearPDEWeak}) with $g(x,t) := f(x,t,u(x,t))$ where $u$ is the solution to our main PDE (\ref{eqn:MainPDE}). Then evidently $v = u$ and its Galerkin approximation $v_h$ satisfies
\begin{equation}
\label{eqn:vhEq}
	(\partial^\alpha_t v_h, \chi) + a(D(t); v_h, \chi) = (f(t,u), \chi), \quad v_h(0) = P_h \varphi, \quad \chi\in V_h.
\end{equation} 
We make the following error decomposition
\begin{equation}
\label{eqn:Decomposition}
	u-u_h = u - v_h + v_h - u_h = v - v_h + v_h - u_h =: e + z.
\end{equation}
In order to find the bound on $z \in V_h$ we start with the error equation
\begin{equation}
\begin{split}
	(\partial^\alpha_t z, \chi) + a(D(t); z, \chi) 
	&= (\partial^\alpha_t v_h, \chi) + a(D(t); v_h, \chi) - (\partial^\alpha_t u_h, \chi) - a(D(t); u_h, \chi) \\
	&= (f(t,u) - f(t,u_h), \chi)
\end{split}
\end{equation}
with $\chi\in V_h$. Here, we have used the equation for $v_h$, that is (\ref{eqn:vhEq}), and (\ref{eqn:MainPDEGalerkin}). Notice that $z(0) = v_h(0) - u_h(0) = P_h\varphi - P_h\varphi = 0$ and by the Lipschitz regularity of $f$ we can estimate further with $\chi = z \in V_h$ and Lemma \ref{lem:DerivativeScalar}
\begin{equation}
	  \|z\|(\partial^\alpha_t \|z\|) + D_- \|\nabla z\|^2 \leq C \|z\| \|u-u_h\|.
\end{equation}
Cancelling by $\|z\|$ leads us to
\begin{equation}
	\partial^\alpha_t \|z\| \leq C \|u-u_h\|,
\end{equation}
whence by application of the fractional integral with $z(0) = 0$ we arrive at
\begin{equation}
	\|z\| \leq C I^\alpha_t\|u-u_h\|. 
\end{equation}
Hence, going back to the decomposition (\ref{eqn:Decomposition}) we obtain
\begin{equation}
	\|u-u_h\| \leq \|e\| + \|z\| \leq \|e\| + C I^\alpha_t \|u-u_h\|.
\end{equation}
As we mentioned above, the estimates on $\|e\|$ are precisely the ones obtained in Theorem \ref{thm:Linear}, therefore
\begin{equation}
	\|u-u_h\| \leq C_\alpha h^2 t^{-\frac{\alpha(2-p)}{2}} \|\varphi\|_p + C I^\alpha_t \|u-u_h\|,
\end{equation}
where by $C_\alpha$ we have denoted the $\alpha$-dependent term in (\ref{eqn:NonsmoothError}). Since $\alpha(2-p)/2 < 1$ the first term on the right-hand side above is $t$-integrable and we can use the fractional Gr\"onwall's inequality (Lemma \ref{lem:Gronwall}). The estimate on the gradient follows from Theorem \ref{thm:Linear} and the inverse inequality since $z\in V_h$, that is
\begin{equation}
	\|\nabla(u-u_h)\| \leq \|\nabla e\| + \|\nabla z\| \leq \|\nabla e\| + h^{-1} \|z\|.
\end{equation}
The proof is finished. 
\end{proof}

\section{Numerical experiments}
Below we present several numerical calculations that support the above stated results. All computations are done in one dimension, that is $\Omega = [0,1]$. Since our analysis concerns semi-discretized version of (\ref{eqn:MainPDE}) in order to compute its solution we also have to discretize the time. We have chosen the L1 scheme for the Caputo derivative \cite{Li19a} with a time step $\Delta t$ and extrapolation in time when calculating the nonlinearities. This method was introduced and analysed in our previous paper concerning the quasilinear subdiffusion equation \cite{plociniczak2021linear}. All our calculations were implemented in the Julia programming language. 

\subsection{Order of convergence}
We first give some verification of the second order of convergence of the scheme (\ref{eqn:MainPDEWeak}). Here, we consider several examples of different nature and smoothness in order to explore various situations. One of them is "manufactured" in order to ascertain the smoothness in time. In all examples we estimate the convergence error $p$ with the use of extrapolation (Aitken's method, see for ex. \cite{plociniczak2021linear}), that is
\begin{equation}
\label{eqn:Aitken}
	p \approx \log_2 \frac{\|u_{h/2} - u_h\|}{\|u_{h/4}-u_{h/2}\|},
\end{equation}
where $u_h$ is a numerical solution for a mesh step $h$. This has an additional advantage to be independent on the time discretization since we are comparing numerical solutions obtained for the same temporal grid spacing $\Delta t$. Unless stated otherwise we compute all solutions at $t=T=1$. Examples in this part are the following. 
\begin{enumerate}
	\item Smooth initial condition with smooth \emph{exact} solution in time
	\begin{equation}
	\label{eqn:Order1}
	\begin{split}
		u(x,t) &= (1+t^2)x(1-x), \\
		D(x,t) &= 1+x+t, \quad f(x,t,u) = (1+t^2) (1+2t+4x) + \frac{2t^{2-\alpha}}{(1+t^2)\Gamma(3-\alpha)}u.
	\end{split}
	\end{equation}
	\item Smooth initial condition with $D, D_x \in L^\infty$, $D_t \in L^1$ but $D_t\notin L^\infty$
	\begin{equation}
	\label{eqn:Order2}
		\varphi(x) = x(1-x), \quad D(x,t) = 1 + \frac{1}{2}\cos(2\pi x) + \sqrt{\left|t-\frac{1}{2}\right|}, \quad f(x,t,u) = u(1-u).
	\end{equation}
	\item Nonsmooth (different boundary conditions) initial condition with $D, D_x \in L^\infty$, $D_t \in L^1$ but $D_t\notin L^\infty$
	\begin{equation}
	\label{eqn:Order3}
		\varphi(x) =  
		\begin{cases}
			0, & 0\leq x < \frac{1}{2}, \\
			1, & \frac{1}{2} \leq x \leq 1,
		\end{cases}
		\quad D(x,t) = 1 + \frac{1}{2}\cos(2\pi x) + \sqrt{\left|t-\frac{1}{2}\right|}, \quad f(x,t,u) = u(1-u).
	\end{equation}
\end{enumerate}

Results of our computations with $\Delta t = 2\times 10^{-3}$ and $h=10^{-2}$ are presented in Tab. \ref{tab:Order}. We have also conducted some simulations for smaller time steps and finer mesh grids with essentially the same results. As we can see, the estimated convergence order computed at a fixed time is manifestly equal to $2$ for all $0<\alpha<1$. There are very slight variations for the example with nonsmooth initial condition but they safely can be considered negligible. Therefore, the presence of a nonlinear source and unbounded $D_t$ have no effect on convergence of the method.  

\begin{table}
	\centering
	\begin{tabular}{rccccc}
		\toprule
		& $\alpha = 0.1$ & $\alpha = 0.25$ & $\alpha = 0.5$ & $\alpha = 0.75$ & $\alpha = 0.9$ \\
		\midrule 
		ex. (\ref{eqn:Order1}) & 1.99 & 1.99 & 1.99 & 1.99 & 1.99 \\
		ex. (\ref{eqn:Order2}) & 2.00 & 2.00 & 2.00 & 2.00 & 2.00 \\
		ex. (\ref{eqn:Order3}) & 1.99 & 1.99 & 1.98 & 1.98 & 1.97 \\
		\bottomrule
	\end{tabular}
	\caption{Estimated orders of convergence for examples (\ref{eqn:Order1}-\ref{eqn:Order3}). Here, the temporal grid was taken to be $\Delta t = 2\times 10^{-3}$ while the spatial mesh $h = 10^{-2}$.}
	\label{tab:Order}
\end{table}

\subsection{Error dependence on time}
Now we proceed to investigation of the temporal dependence of the error as in (\ref{eqn:NonsmoothErrorQuasilinear}). Similarly as above we consider several examples. In each of them the exact solution is not available and in order to estimate the error we compute a reference solution for a very fine grid. This has to be taken with care since we do not want the time discretization error to overcome our calculations. According to our knowledge, there are no literature results concerning the time-dependent fully-discrete schemes for semilinear subdiffusion equation. However, we can make an educated guess of the method to act similarly to others. For the L1 method we expect that the full discretization error should behave according to \cite{al2019numerical} similarly as in the case of time-independent diffusivity in a semilinear equation
\begin{equation}
	\|u(t_n) - u^n\| \leq C \left(t^{-\frac{\alpha(2-p)}{2}}h^2 + t_n^{\frac{\alpha p}{2} -1} \Delta t\right)
\end{equation}
Therefore, in order to make the temporal part vanishingly small for all times we should take the smallest $n$ possible, that is $t_1 = \Delta t$. Then, the requirement that the first term above dominates the second is given by
\begin{equation}
\label{eqn:DeltaT}
	\Delta t < h^\frac{2}{\alpha},
\end{equation} 
which could be prohibitively small for $\alpha\rightarrow 0^+$. We have thus decided to consider only $\alpha \geq 0.5$ and plan to conduct thorough analytical and numerical studies of the fully discrete schemes for the time-dependent diffusivity case in the future. All simulations were conducted with $h=0.5\times 10^{-2}$ and $\Delta t$ according to the above formula. In the following examples we always choose the diffusivity and the source as in examples (\ref{eqn:Order2}-\ref{eqn:Order3}) and focus only on smoothness of the initial condition. 
\begin{enumerate}
	\item Smooth initial condition ($p=2$)
	\begin{equation}
	\label{eqn:ErrTime1}
		\varphi(x) = x(1-x) \in \dot{H}^2 = H^2 \cap H_0^1.
	\end{equation}
	\item Initial condition of intermediate smoothness ($p = 3/2$)
	\begin{equation}
	\label{eqn:ErrTime2}
		\varphi(x) = 1-\left|2x-1\right| \in \dot{H}^{\frac{3}{2}-\epsilon}.
	\end{equation}
	\item Initial condition of intermediate smoothness ($p = 1$)
	\begin{equation}
	\label{eqn:ErrTime3}
		\varphi(x) = \sqrt{x(1-x)} \in \dot{H}^{1-\epsilon}.
	\end{equation}
	\item Nonsmooth initial condition ($p = 1/2$)
	\begin{equation}
	\label{eqn:ErrTime4}
		\varphi(x) = \begin{cases}
			0, & 0\leq x < \frac{1}{2}, \\
			1, & \frac{1}{2} \leq x \leq 1;
		\end{cases} \in \dot{H}^{\frac{1}{2} - \epsilon},
	\end{equation}
\end{enumerate}
where $\epsilon>0$ is arbitrary. The check to what space the initial condition belongs is done by observing that the eigenvalues of the negative Dirichlet Laplacian on an interval are $O(n^2)$ when $n\rightarrow \infty$. Additionally, Fourier coefficients of $\varphi$ can be calculated exactly to determine their asymptotic behaviour. For example, when $\varphi(x) = \sqrt{x(1-x)}$ we obtain that its orthogonal expansion coefficients are proportional to $J_1(n \pi/2) n^{-1} = O(n^{-3/2})$ for large $n$. Therefore, the norm (\ref{eqn:HDotNorm}) is finite for $s < 1$. 

According to (\ref{eqn:NonsmoothErrorQuasilinear}) the error versus time plotted on the log-log scale should become parallel to the line with a tangent $-\alpha(2-p)/2$ at least for $t\rightarrow 0^+$. In other words, $t^{\alpha(2-p)/2}$ multiplied by the error should be time-independent. This phenomenon is presented in Fig. \ref{fig:TimeDependence} for $\alpha = 0.75$, however, our calculations showed the same behaviour for different values of $\alpha$. The smoothness degree $p$ of the initial condition corresponds to examples (\ref{eqn:ErrTime1}-\ref{eqn:ErrTime4}). Some quantitative results are presented also in Tab. \ref{tab:ErrTime}. They were obtained by computing the solution to each one of our examples for two different values of $\alpha$ with $h=10^2$, $T = 10^{-2}$, and $\Delta t$ computed according to (\ref{eqn:DeltaT}). Then, the nonlinear regression was applied to the error as a function of time to fit the first $10^2$ error points to $a t^{-s}$. The exponent $s$ has then been put in the table. We can see that the agreement with the theory is decent especially for the larger value of $\alpha$. Calculations for smaller $\alpha$ require very fine time steps in the L1 scheme setting. Using a higher order method or non uniform time grid should make these calculations more feasible. We have not gone into that direction since fully discrete schemes for the investigated problem are not yet fully understood. Nevertheless, our initial simulations show a clear support for the above proved results.  

\begin{figure}
	\centering
	\includegraphics[scale = 0.8]{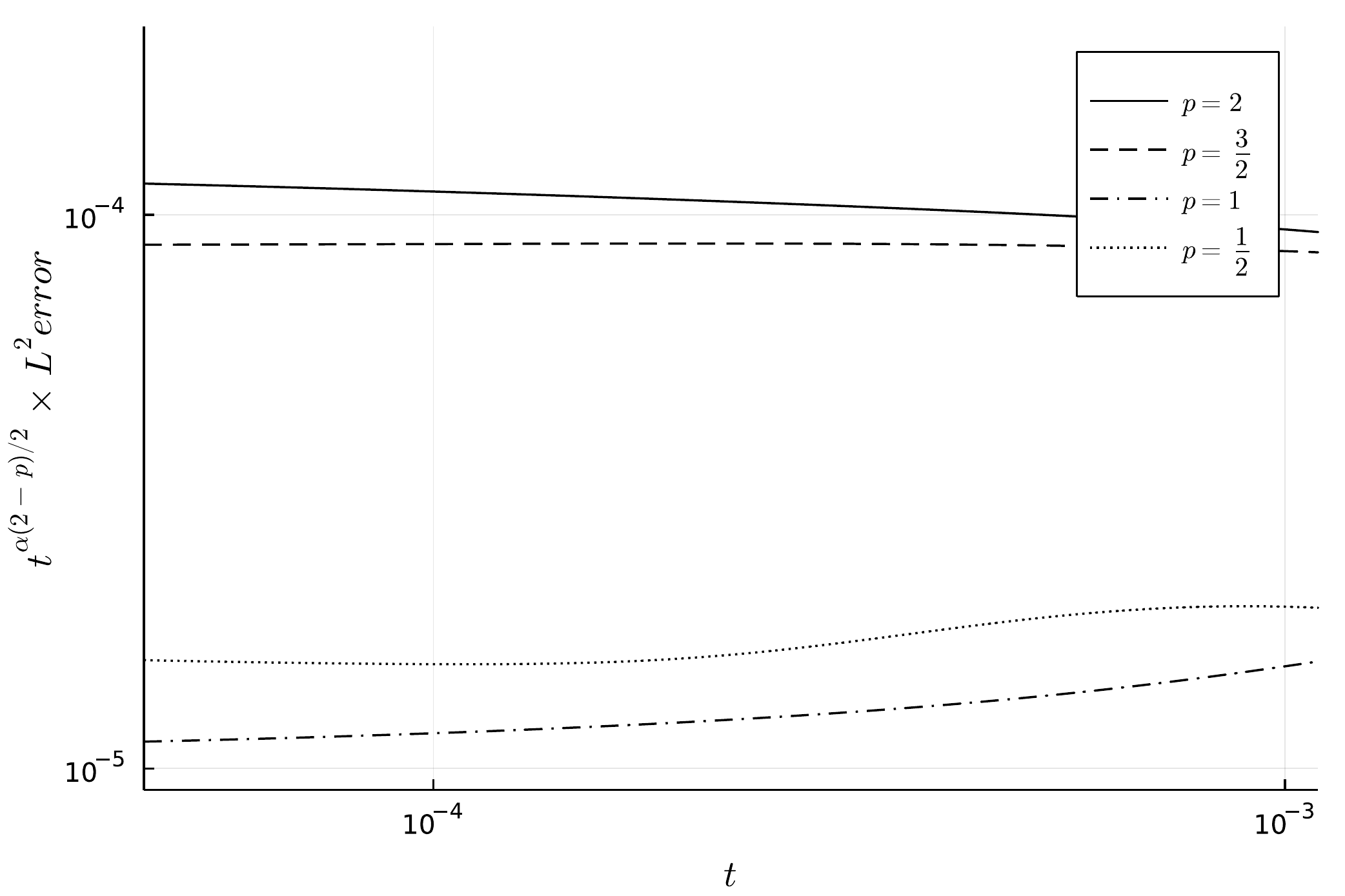}
	\caption{The graph of $t^{\alpha(2-p)/2} \times $ ($L^2$ error) plotted with respect to time on a log-log scale. Here, we have taken $\alpha = 0.75$ and indicated the smoothness $p$ of the initial condition in the legend. }
	\label{fig:TimeDependence}
\end{figure}

\begin{table}
	\centering
	\begin{tabular}{rcc}
		\toprule
		& $\alpha = 0.5$ & $\alpha = 0.75$ \\
		\midrule 
		ex. (\ref{eqn:ErrTime1}) & 0.05 (0.00) & 0.04 (0.00) \\
		ex. (\ref{eqn:ErrTime2}) & 0.17 (1.13) & 0.19 (0.19) \\
		ex. (\ref{eqn:ErrTime3}) & 0.25 (0.25) & 0.33 (0.35) \\
		ex. (\ref{eqn:ErrTime4}) & 0.32 (0.37) & 0.56 (0.56) \\
		\bottomrule
	\end{tabular}
	\caption{Estimated exponent $s$ in the fitted error function $a t^{-s}$ as in (\ref{eqn:NonsmoothErrorQuasilinear}) for examples (\ref{eqn:ErrTime1}-\ref{eqn:ErrTime4}). Here, the temporal grid was taken according to (\ref{eqn:DeltaT}) with $\gamma = 10^{-1}$ while the spatial mesh $h = 10^2$. The number in parenthesis is the exact value $\alpha(2-p)/2$ rounded to two decimal places.}
	\label{tab:ErrTime}
\end{table}

\section{Conclusion}
By using the energy method coupled with two generalizations of the classical Gr\"onwall's inequality we have been able to significantly relax assumptions on the diffusivity $D(x,t)$ and show $L^2$ error estimates of the Galerkin method in the case of nonsmooth data. The linear result could then be carried over into the semilinear case with the use of frozen nonlinearity technique and a recently proved lemma concerning scalar product and the fractional derivative. In our further work we would proceed to the quasilinear subdiffusion equations, that is when the diffusivity is a function of the solution $u$. This example is sufficiently general to model many real-world phenomena for instance, these found in hydrology. This problem could produce interesting questions concerning regularity requirements and estimates on $\|\nabla u_h\|$. Our numerical calculations conducted to support the theory indicated a further need of theoretical analysis of the fully discrete schemes for the time-dependent nonlinear subdiffusion equation. 

\section*{Acknowledgement}
Ł.P. has been supported by the National Science Centre, Poland (NCN) under the grant Sonata Bis with a number NCN 2020/38/E/ST1/00153.

\bibliography{biblio}
\bibliographystyle{plain}

\end{document}